\documentclass{amsart}
\usepackage{amssymb}
\usepackage{comment}
\usepackage{color}
\usepackage{bm}
\newtheorem{theorem}{Theorem}[section]
\newtheorem{lemma}[theorem]{Lemma}

\theoremstyle{definition}
\newtheorem{definition}[theorem]{Definition}

\newtheorem{question}[theorem]{Question}
\theoremstyle{remark}

\numberwithin{equation}{section}

\begin{document}

\title[one-sided subshifts]{The conjugacy relation on one-sided subshifts is non-treeable}


%    author one information

\author{Ruiwen Li}
\address{School of Mathematical Sciences and LPMC, Nankai University, Tianjin 300071, P.R. China}
\curraddr{}
\email{rwli@mail.nankai.edu.cn}
\thanks{The author would like to thank Su Gao, Bo Peng and Marcin Sabok for their valuable advice. The author acknowledges the partial support of his research by the National Natural Science Foundation of China (NSFC) grant 124B2001. This work was done when the author visiting McGill University, the author acknowledges the support by McGill University.}
\date{\today}

\dedicatory{}

%    "Communicated by" -- provide editor's name; required.
\commby{}

%    Abstract is required.

\begin{abstract}
 In this paper we study the conjugacy relation on one-sided subshifts in the viewpoint of descriptive set theory. We show the conjugacy relation on one sided subshifts with the alphabet set $\{0,1\}$ is non-treeable and non-amenable.
\end{abstract}

\maketitle

\section{Introduction}
The descriptive set-theoretic complexity theory for equivalence relations on standard Borel spaces was developed by Becker, Kechris, Hjorth and others (see \cite{Becker}, \cite{Hjorth}). This theory gives us a framework to compare the complexity of different equivalence relations, from descriptive set theory or from other aspects of mathematics. Let $E$ and $F$ be equivalence relations defined on standard Borel spaces $X$ and $Y$, respectively. Then we say $E$ is \textbf{Borel reducible} to $F$, denoted by $E\le_B F$, if there is a Borel-measurable map $f$ from $X$ to $Y$ such that for every $x,y\in X$, we have $$xEy\iff f(x)Ff(y).$$ If $E\le_B F$, then we regard $F$ as a more complicated equivalence relation than $E$ under this framework. When $E\le_B F$ and $F\le_B E$, we say $E$ is \textbf{Borel bireducible} with $F$, denoted by $E\sim_B F$. If $E\sim_B F$, then we regard that $E$ and $F$ have the same complexity.

In this paper we talk about the conjugacy relation on dynamical systems. A (topological) \textbf{dynamical system} is a pair $(X,T)$ where $X$ is a compact metrizable space and $T$ is a continuous map from $X$ to $X$. A dynamical system $(X,T)$ is \textbf{conjugate} to another dynamical system $(Y,S)$ if there is a homeomorphism map $f$ from $X$ to $Y$ such that $fT=Sf$, where $f$ is called a \textbf{conjugacy map}. A dynamical system $(X,T)$ is \textbf{invertible} if $T$ is a homeomorphism from $X$ to $X$. For a dynamical system $(X,T)$ and $x\in X$, the \textbf{orbit} of $x$ is the set $\{T^n(x):n\in\mathbb{N}\}$, denoted by $\mathcal{O}(x)$. A dynamical system $(X,T)$ is \textbf{transitive} if there is a point $x\in X$ whose orbit is dense in $X$. A dynamical system $(X,T)$ is \textbf{minimal} if for every point $x\in X$, $\mathcal{O}(x)$ is dense in $X$.

Classifying dynamical systems by isomorphism or conjugacy is an extensively studied topic in descriptive set theory. Hjorth \cite{Hjorth_1} proved the isomorphism relation on measure preserving transformations is not Borel. For ergodic measure preserving transformations, Foreman and Weiss \cite{FW} showed the isomorphism relation on ergodic measure preserving transformations is not classifiable by countable structures, then Foreman, Rudolph and Weiss \cite{FRW} showed this relation is not Borel. In this paper we focus on topological dynamical systems. Li and Peng \cite{Oxtoby} proved that the conjugacy relation on invertible minimal systems is not classifiable by countable structures. We often study dynamical systems with a specific compact metric space. For Cantor systems, Carmelo and Gao \cite{CarmeloGao} proved that the conjugacy relation on invertible Cantor systems is Borel bireducible with the isomorphism relation of countable graphs. Kaya \cite{Kaya} proved that $=^+$ is Borel reducible the conjugacy relation on invertible Cantor minimal systems. And a recent result by Deka, Garc\'{i}a-Ramos, Kasprzak, Kunde and Kwietniak \cite{NonBorel} shows that the conjugacy relation on invertible Cantor minimal systems is not Borel. For interval dynamical systems, Bruin and Vejnar \cite{BV} showed that the conjugacy relation on invertible interval systems is Borel bireducible with the isomorphism relation of countable graphs. Finally for 2-torus, Peng showed that both the conjugacy relation on invertible minimal 2-torus systems \cite{Bo_1} and the conjugacy relation on 2-torus diffeomorphisms \cite{Bo} are not classifiable by countable structures. All of these equivalence relations are somehow very complicated relations in descriptive set theory.

The complexity of the conjugacy relation on subshifts and its subclasses is an active subarea in descriptive dynamics, which is different from the classes we mentioned before. A \textbf{countable equivalence relation} is an equivalence relation whose every equivalence class is a countable set. By a famous theorem of Curtis, Hedlund and Lyndon \cite{BC}, the conjugacy relation on subshifts is a countable Borel equivalence relation. A countable Borel equivalence relation is \textbf{universal} if every countable Borel equivalence relation is Borel reducible to it. We have a universal countable equivalence relation, denoted by $E_\infty$ (see \cite[section 7.3]{main}). The equivalence relation $E_\infty$ has much lower complexity than what we mentioned in the last paragraph. Clemens \cite{Clemens} proved that the the conjugacy relation on subshifts is Borel bireducible with $E_\infty$. By \cite{AK}, the Borel reduction on countable Borel equivalence relation is still very complicated. We have some properties to judge the complexity of a countable Borel equivalence relation. If an equivalence relation is Borel reducible to the identity relation on $\mathbb{R}$, we say it is \textbf{smooth}. The equivalence relation $E_0$ is not smooth, and countable equivalence relations that are Borel reducible to it are called \textbf{hyperfinite} relations. Every hyperfinite equivalence relation is \textbf{amenable} and \textbf{treeable}. Thomas \cite{Thomas} showed that the conjugacy relation on Toeplitz subshifts is not smooth. Gao, Li, Peng and Sun \cite{rank2} showed that the conjugacy relation on topological rank 2 Toeplitz subshifts is hyperfinite. Deka, Kwietniak, Peng and Sabok \cite{paper} proved that the conjugacy relation on subshifts with the specification property is non-treeable and non-amenable.

However, as a natural analogue of subshifts and an important example of dynamical systems, we know very little about the complexity of the conjugacy relation on one-sided subshifts. In fact, until now we do not know the exact complexity of the conjugacy relation on the full class of one-sided subshifts. At the end of \cite{Clemens}, Clemens asked this question.

\begin{question} 
    (\cite{Clemens}) What is the complexity of the isomorphism of one-sided shifts on $\{0,1\}^\mathbb{N}$?
\end{question}

By the proof of \cite[Theorem 6.4]{paper}, we can get that the conjugacy relation on transitive one-sided subshifts with an alphabet set large enough is non-treeable. But here, not as the two-sided case, it's nontrivial to shrink the alphabet set, we will do some modifications to the construction in \cite{paper}. Then we can partially answer the question.

\begin{theorem}
    The conjugacy relation on transitive one-sided subshifts with the alphabet set $\{0,1\}$ is a non-treeable and non-amenable countable Borel equivalence relation.
\end{theorem}

\section{Preliminaries}

In this paper we always let $A$ be a finite alphabet set such that $|A|\ge2$. Denote the set of all finite words of $A$ by $A^*$, in other words, $A^*=A^{<\omega}$. Denote the set of all finite words of $A$ whose length are odd by $A^*_{\rm odd}$. For $x\in A^*$ or $x\in A^\mathbb{N}$ or $x\in A^\mathbb{Z}$, if $a< b$ such that $a$ and $b-1$ are in the domain of $x$, then by $x[a,b)$ we mean the unique word $w\in A^*$ such that $|w|=b-a$ and $w(i)=x(a+i)$ for every $0\le i<b-a$. For $w,v\in A^*$, we denote the concatenation of $w$ and $v$ by $wv$, in other words, $$wv\in A^*,|wv|=|w|+|v|,wv[0,|w|)=w\;\,{\rm and}\;\,wv[|w|,|w|+|v|)=v.$$

%By a \textbf{dynamical system} we mean a pair $(X,T)$ where $X$ is a compact metric space and $T$ is an automorphism on $X$. For dynamical systems $(X,T)$ and $(Y,S)$, we say $(X,T)$ and $(Y,S)$ are \textbf{conjugate} if there is a homeomorphism $f$ from $X$ to $Y$ such that $fT=Sf$. Here $f$ is called a \textbf{conjugacy map} from $(X,T)$ to $(Y,S)$.

%Similarly, a \textbf{semi dynamical system} is a pair $(X,T)$ such that $X$ is a compact metric space and $T$ is a continuous map from $X$ to $X$. Semi dynamical systems $(X,T)$ and $(Y,S)$ are \textbf{conjugate} if there is a homeomorphism $f$ from $X$ to $Y$ such that $fT=Sf$. We still call $f$ a \textbf{conjugacy map} from $(X,T)$ to $(Y,S)$.

%For a dynamical system or semi dynamical system $(X,T)$ and $x\in X$, the \textbf{orbit} of $x$ is the set $$\{T^n(x):n\in\mathbb{N}\},$$ denoted by $\mathcal{O}(x)$. A dynamical system or semi dynamical system $(X,T)$ is \textbf{transitive} if there is $x\in X$ such that $\mathcal{O}(x)$ is dense in $X$.

A canonical example for invertible dynamical system is subshift. The \textbf{Bernoulli shift} is a dynamical system $(A^{\mathbb{Z}},S)$ with the compact metric space $A^\mathbb{Z}$ where $A$ is a finite alphabet set, and with the left shift action $S$ such that $$\forall n\in\mathbb{Z}\forall x\in A^\mathbb{Z}\; S(x)(n)=x(n+1).$$ A \textbf{subshift} with the alphabet set $A$ is a closed subset $X$ of $A^\mathbb{Z}$ such that $S(X)=X$. We view a subshift $X$ as the invertible dynamical system $(X,S)$.

We have analogous concepts for non-invertible systems. The \textbf{one-sided full shift} is a dynamical system $(A^{\mathbb{N}},S)$ with the compact metric space $A^\mathbb{N}$ where $A$ is a finite alphabet set, and with the left shift action $S$ such that $$\forall n\in\mathbb{N}\forall x\in A^\mathbb{Z}\; S(x)(n)=x(n+1).$$Note that when $|A|\ge2$, the left shift action on the one-sided full shift is not invertible. A \textbf{one-sided subshift} with the alphabet set $A$ is a closed subset $X$ of $A^\mathbb{N}$ such that $S(X)\subset X$. Then a one-sided subshift $X$ can be regarded as the dynamical system $(X,S)$.

To study the conjugacy relation on subshifts and one sided subshifts in the viewpoint of descriptive set theory, we should give the class of (one-sided) subshifts a standard Borel structure. For a Polish space $X$, the Vietoris topology on the class of all compact subsets of $X$, denoted by $K(X)$, is a Polish topology (see \cite[Section 4.F]{Kechris}). The set of all subshifts with finite alphabet set $A$ is a $G_\delta$ subset of $K(A^\mathbb{Z})$, and the set of all one-sided subshifts with finite alphabet set $A$ is a $G_\delta$ subset of $K(A^\mathbb{N})$ (see \cite[Fact 2.1]{Oxtoby}). This gives the class of (one-sided) subshifts a Polish topology.

For the conjugacy relation on subshifts, note the following theorem.
\begin{theorem}\!\!\mbox{{\rm (Curtis–Hedlund–Lyndon, \cite{BC})}}
     Let $A$ be a finite alphabet set, let $X,Y \subset A^\mathbb{Z}$
be subshifts, and let $\varphi: X \to Y$ be a conjugacy map from $(X,S)$ to $(Y,S)$. Then there
exists $n\in \mathbb{N}$ and a function $C: A^{2n+1} \to A$ such that for all $x \in X$ and $i \in\mathbb{Z}$,
$$\varphi(x)(i) = C(x[i-n,i +n+1)).$$
\end{theorem}

Here $C$ is called a \textbf{block code} of $f$. Since $f$ is totally determined by its block code, and there are only countably many block codes, then we have that the conjugacy relation on subshifts is a countable equivalence relation. Also by the statement of this theorem, the conjugacy relation on subshifts is Borel. Similarly the conjugacy relation on one-sided subshifts is a countable Borel equivalence relation.

Let $E$ be a countable Borel equivalence relation on the standard Borel space $X$. We say
that $E$ is \textbf{hyperfinite} if $E=\bigcup_{n\in\mathbb{N}}E_n$ such that for every $n\in\mathbb{N}$, $E_n$ is a Borel equivalence relation whose equivalence classes are all finite, and $E_n\subset E_{n+1}$. The
equivalence relation $E$ is \textbf{treeable} \cite[Definition 3.1]{JKL} if there exists a Borel
acyclic graph $G$ on the $X$ such that for any $x,y\in X$, we have that $xEy$ if and only if $x$ and $y$ are in the same connected $G$ component. And we say $E$ is \textbf{amenable} \cite[Section 7.4] {main} if we have a sequences of Borel functions $f_n$ from $E$ to $\mathbb{R}_{\ge0}$ such that:
\begin{enumerate}
    \item[(a)] for any $x\in X$ and $n\in\mathbb{N}$, $\Sigma_{yEx}f_n(x,y)=1$;
    \item[(b)] for any $x,x'\in X$, if $xEx'$, then we have that $${\rm lim}_{n\to \infty}\Sigma_{yEx}|f_n(x,y)-f_n(x',y)|=0.$$
\end{enumerate}

A classical example of hyperfinite equivalence relation is $E_0$ defined on $\{0,1\}^\mathbb{N}$. For $x,y\in \{0,1\}^\mathbb{N}$, let $xE_0y$ if for all $n\in\mathbb{N}$ large enough, we have that $x(n)=y(n)$. Every hyperfinite equivalence relation is Borel reducible to $E_0$. Also every hyperfinite equivalence relation is treeable and amenable.

 \section{Proof of the Main Theorem}

Throughout this section let $A$ be the alphabet set with distinct symbols $a_{i}$, $\tilde{a}_i$ and $\#$ where $1\le i\le6$, in other words, $$A=\{a_i:1\le i\le6\}\cup\{\tilde{a}_i:1\le i\le6\}\cup\{\#\}.$$

The following notations are from \cite{paper}.

Given $R\subset (A\setminus \{\#\})_{\rm odd}^*$, we denote the subshift with forbidden words set $\{\#w\#:w\in R\}$ by $X(R)$, in other words, $$X(R)=\{x\in A^\mathbb{Z}: \forall\, w\in R \quad\#w\#\;is \,not\;a\;subword\;of\;X\}.$$ Note that $(A\setminus \{\#\})^\mathbb{Z}\subset X(R)$, so $X(R)$ is a nonempty subshift. Let $$\mathcal{S}(A)=\{X(R):R\subset (A\setminus \{\#\})_{\rm odd}^*\}.$$ Identify a subset of $(A\setminus \{\#\})_{\rm odd}^*$ and its characteristic function, then the map that sends $R\in \{0,1\}^{(A\setminus \{\#\})_{\rm odd}^*}$ to $X(R)$ is continuous from $\{0,1\}^{(A\setminus \{\#\})_{\rm odd}^*}$ to $\mathcal{S}(A)$ with the Vietoris topology. So $\mathcal{S}(A)$ is a compact metric space.

Let ${\rm Aut}(A^\mathbb{Z})$ be the group of conjugacy maps from $(A^\mathbb{Z},S)$ to $(A^\mathbb{Z},S)$. For $w\in A^*$, we denote the cylinder set $\{x\in A^\mathbb{Z}:x[0,|w|)=w\}$ by $[w]$.

\begin{definition}[\cite{paper}]\label{def}
    We say that $f\in {\rm Aut}(A^\mathbb{Z})$ is \textbf{$\#$-preserving} if there is a bijection
$f^*: (A\setminus\{\#\})^* \to (A\setminus\{\#\})^*$
that preserves the word length and such that for every $w \in (A\setminus\{\#\})^*$ we have
$$f([\#w\#]) = [\#f^*(w)\#].$$
We say that the corresponding length-preserving bijection $f^*$ \textbf{represents} $f$ on
$(A\setminus\{\#\})^*$. We write ${\rm Aut}(A^\mathbb{Z},\#)$ for the set of all $\#$-preserving automorphisms
$f\in {\rm Aut}(A^\mathbb{Z})$.
\end{definition}

By definition, ${\rm Aut}(A^\mathbb{Z},\#)$ is a subgroup of ${\rm Aut}(A^\mathbb{Z})$. The group ${\rm Aut}(A^\mathbb{Z},\#)$ naturally acts on $\mathcal{S}(A)$, \cite[proposition 5.3]{paper} says that, for $f\in{\rm Aut}(A^\mathbb{Z},\#)$ and $R\subset (A\setminus \{\#\})_{\rm odd}^*$, take the bijection $f^*$ representing $f$, then we have $f(X(R))=X(f^*(R))$. Take $$f\cdot X(R)=f(X(R))=X(f^*(R)),$$we get a continuous group action of ${\rm Aut}(A^\mathbb{Z},\#)$ on $\mathcal{S}(A)$.

We will use the following definition and lemma in our proof.

\begin{definition}[\cite{paper}]
     We say that $f\in{\rm Aut}(A^\mathbb{Z},\#)$ is \textbf{almost trivial} if $f^*(w) = w$ for all but finitely many $w \in (A\setminus\{\#\})^*$, where $f^*$ represents $f$.
\end{definition}

\begin{theorem}[\cite{paper}]\label{1}
     
     \begin{enumerate}
         \item The induced action of ${\rm Aut}(A^\mathbb{Z},\#)$ on $\mathcal{S}(A)$ preserves the conjugacy relation.
         \item If $\Gamma$ is a countable subgroup of ${\rm Aut}(A^\mathbb{Z},\#)$ which does not contain an almost
trivial non-identity element, then the induced action on $\mathcal{S}(A)$ preserves a probability
measure and is a.e. free.
     \end{enumerate}
\end{theorem}

Now we construct a countable subgroup $\Gamma$ of ${\rm Aut}(A^\mathbb{Z},\#)$. For $1\le i\le3$, let $g_i$ be the function from $A^2$ to $A$ such that
\begin{equation}\label{g}
    g_i(bc)= \left \{\begin{array}{lr}  a_i,&b=\tilde{a}_i;c=a_j\;where\;1\le j\ne i\le3,\;or\;c=\#.\\\tilde{a}_i, &b={a}_i;c=a_j\;where\;1\le j\ne i\le3,\;or\;c=\#.\\ b, &otherwise. \end{array} \right.
\end{equation}
And for $4\le i\le6$, let $g_i$ be the function from $A^2$ to $A$ such that$$g_i(bc)= \left \{\begin{array}{lr}  a_i,&b=\tilde{a}_i;c=a_j\;where\;4\le j\ne i\le6,\;or\;c=\#.\\\tilde{a}_i, &b={a}_i;c=a_j\;where\;4\le j\ne i\le6,\;or\;c=\#.\\ b, &otherwise. \end{array} \right.$$ For $1\le i\le6$, let $f_i$ be a map from $A^\mathbb{Z}$ to $A^\mathbb{Z}$ such that 
\begin{equation}\label{f}
    f_i(x)(k)=g_i(x(k)x(k+1)).
\end{equation}

\begin{lemma}\label{order 2}
    For every $1\le i\le6$, $f_i\in {\rm Aut}(A^\mathbb{Z},\#)$ and $f^2_i={\rm id}$.
\end{lemma}
\begin{proof}
    By the definition of $f_i$, the function $f_i$ is continuous and commutes with the shift action $S$. Without loss of generality assume that $1\le i\le3$. For $x\in A^\mathbb{Z}$ and $k\in\mathbb{Z}$, if $x(k)\ne a_i,\tilde{a}_i$, by the definition of $g_i$ we have that $f^2_i(x)(k)=f_i(x)(k)=x(k)$. If $x(k)=a_i$, we have 2 cases:
    
    \noindent\textbf{Case 1}. $x(k+1)=a_j$ where $1\le j\ne i\le3$, or $x(k+1)=\#$.
    
    Since $x(k+1)\ne a_i,\tilde{a}_i$, we have $f_i(x)(k+1)=x(k+1)$. By definition $f_i(x)(k)=\tilde{a}_i$. So by the definition of $g_i$, $f^2_i(x)(k)=a_i=x(k)$

    \noindent\textbf{Case 2}. Otherwise.
    
    By the definition of $f_i$ and $g_i$, $f_i(x)(k)=a_i$. If $x(k+1)\ne a_i,\tilde{a}_{i}$, then $f_i(x)(k+1)=x(k+1)$ and we have that $f^2_i(x)(k)=x(k)$. If $x(k+1)= a_i$ or $x(k+1)=\tilde{a}_{i}$, then by the definition of $g_i$, $f_i(x)(k+1)=a_i$ or $f_i(x)(k+1)=\tilde{a}_{i}$. So we get that $f^2_i(x)(k)=a_i=x(k)$ by the definition of $g_i$. This ends Case 2.

    Similarly, when $x(k)=\tilde{a}_i$, we have that $f^2_i(x)(k)=x(k)$. By the arbitrariness of $x$ and $k$, $f^2_i={\rm id}$. This also shows that $f_i$ is a bijection from $A^\mathbb{Z}$ to itself, so $f_i\in{\rm Aut}(A^\mathbb{Z})$.

    Let $f^*_i$ be the map from $(A\setminus{\#})^*$ to $(A\setminus{\#})^*$ such that $f^*_i(\emptyset)=\emptyset$ and $f^*_i(a_0\cdots a_n)=g_i(a_0a_1)g_i(a_1a_2)g_i(a_2a_3)\cdots g_i(a_n\#)$. Then by the definition of $f_i$, the map $f^*_i$ represents $f_i$, so $f_i\in {\rm Aut}(A^\mathbb{Z},\#)$.
\end{proof}

Denote the subgroup of ${\rm Aut}(A^\mathbb{Z},\#)$ generated from $\{f_i :1\le i\le 6\}$ by $\Gamma$ .

\begin{theorem}
    The group $\Gamma$ is isomorphic to $(\mathbb{Z}_2*\mathbb{Z}_2*\mathbb{Z}_2)^2$, and there is no almost
trivial non-identity element in $\Gamma$.
\end{theorem}

\begin{proof}
    Firstly we show that for $1\le i\le3$ and $4\le j\le 6$, we have that $f_if_j=f_jf_i$. Denote $\{\#\}\cup\{a_{i'}:1\le i'\ne i\le3\}$ by $B_i$, denote $\{\#\}\cup\{a_{j'}:4\le j'\ne j\le6\}$ by $B_j$, note that $B_i$ and $B_j$ are invariant sets under $f_i$ and $f_j$. Let $x\in A^\mathbb{Z}$ and $k\in\mathbb{Z}$, we have 3 cases:

    \noindent\textbf{Case 1}. $x(k)\ne a_i,\tilde{a}_i,a_j,\tilde{a}_j$.

    In this case we always have $f_if_j(x)(k)=f_jf_i(x)(k)=x(k)$.

    \noindent\textbf{Case 2}. $x(k)= a_i$ or $x(k)=\tilde{a}_i$.

    Without loss of generality assume that $x(k)=a_i$. If $x(k+1)\in B_i$, then $f_i(x)(k)=\tilde{a}_i$, and $f_jf_i(x)(k)=\tilde{a}_i$ for $j\ne i$; we also have that $f_j(x)(k)=a_i$ and $f_j(x)(k+1)\in B_i$ by the invariance of $B_i$ under $f_j$, then $f_if_j(x)(k)=\tilde{a}_i=f_jf_i(x)(k)$. If $x(k+1)\notin B_i$, then $f_i(x)(k)={a}_i$, and $f_jf_i(x)(k)={a}_i$ for $j\ne i$; we also have that $f_j(x)(k)=a_i$ and $f_j(x)(k+1)\notin B_i$ by the invariance of $B_i$ under $f_j$, then $f_if_j(x)(k)={a}_i=f_jf_i(x)(k)$.

     \noindent\textbf{Case 3}. $x(k)= a_j$ or $x(k)=\tilde{a}_j$.

     Similar to Case 2, $f_if_j(x)(k)=f_jf_i(x)(k)$. 
     
     By the arbitrariness of $x$ and $k$, we have that $f_if_j=f_jf_i$. 
     
     Then to show $\Gamma$ is isomorphic to $(\mathbb{Z}_2*\mathbb{Z}_2*\mathbb{Z}_2)^2$, we just need to show for every finite sequence $i_1\cdots i_n\in \{1,2,3\}^*$ and $j_1\cdots j_m\in \{4,5,6\}^*$ such that $m+n\ge1$ and for every $1\le k<n$, $i_k\ne i_{k+1}$, for every $1\le k<m$, $j_k\ne j_{k+1}$, we have that $f_{i_1}\cdots f_{i_n}f_{j_1}\cdots f_{j_m}\ne{\rm id}$. Denote $f_{i_1}\cdots f_{i_n}f_{j_1}\cdots f_{j_m}$ by $f$.\\

      \noindent\textbf{Claim 1}.   For infinitely many words $v$, $f^*_{i_1}\cdots f^*_{i_n}f^*_{j_1}\cdots f^*_{j_m}(v)\ne v$.\\
     
\begin{proof}
      Without loss of generality assume that $n\ne 0$. Since for $1\le i\le3$ and $4\le j\le 6$, we have $f_if_j=f_jf_i$, then we also have that $f^*_if^*_j=f^*_jf^*_i$. For every word $w\in \{a_1,a_2,a_3\}^*$, inductively it's easy to get that the $(|w|+n-k)$'th position of $f^*_{i_k}\cdots f^*_{i_n}(w\tilde{a}_{i_1}\cdots \tilde{a}_{i_n})$ is $a_{i_{n-k+1}}$ for every $1\le k\le n$, so the $|w|$'th position of $f^*_{i_1}\cdots f^*_{i_n}(w\tilde{a}_{i_1}\cdots \tilde{a}_{i_n})$ is $a_{i_1}$. The word $w\tilde{a}_{i_1}\cdots \tilde{a}_{i_n}$ is invariant under $f^*_4,f^*_5,f^*_6$, so the $|w|$'th position of $f^*_{i_1}\cdots f^*_{i_n}f^*_{j_1}\cdots f^*_{j_m}(w\tilde{a}_{i_1}\cdots \tilde{a}_{i_n})$ is $a_{i_1}$, which implies that $f^*_{i_1}\cdots f^*_{i_n}f^*_{j_1}\cdots f^*_{j_m}(w\tilde{a}_{i_1}\cdots \tilde{a}_{i_n})\ne w\tilde{a}_{i_1}\cdots \tilde{a}_{i_n}$. This ends the proof of the claim.
      \end{proof}

      By Definition \ref{def} it's routine to check that $f^*_{i_1}\cdots f^*_{i_n}f^*_{j_1}\cdots f^*_{j_m}$ represents $f$, by Claim 1 we have that $f^*_{i_1}\cdots f^*_{i_n}f^*_{j_1}\cdots f^*_{j_m}\ne{\rm id}$, so we get that $f_{i_1}\cdots f_{i_n}f_{j_1}\cdots f_{j_m}\ne{\rm id}$. Also by the claim, $f$ is not an almost
trivial element. This ends the proof.

\end{proof}

 Let $E$ be the equivalence relation induced by action of $\Gamma$ on $\mathcal{S}(A)$. 
 
 \begin{theorem}
     The equivalence relation $E$ is non-treeable and non-amenable.
 \end{theorem}
 \begin{proof}

 By Theorem \ref{1} the action of $\Gamma$ on $\mathcal{S}(A)$ preserves a probability
measure and is a.e. free. By the result of Pemantle–Peres \cite{PP}, an equivalence relation induced by a measure preserving free action of $F_2\times F_2$ is not treeable. Since $\Gamma$ is isomorphic to $(\mathbb{Z}_2*\mathbb{Z}_2*\mathbb{Z}_2)^2$, which contains a copy $F_2\times F_2$ of as a subgroup, then we have some Borel equivalence relation $F\subset E$ that is not treeable. By \cite[Theorem 3.3(iii)]{JKL}, $E$ is non-treeable.

The group $\Gamma$ is not an amenable group, then by \cite[Theorem 7.4.8]{main}, $E$ is non-amenable.

 \end{proof}
 \begin{theorem}
     The conjugacy relation on transitive one-sided subshifts with the alphabet set $\{0,1\}$ is non-treeable and non-amenable.
 \end{theorem}
 \begin{proof}
     By \cite[Theorem 3.3(iii)]{JKL} and \cite[Theorem 2.15(iii)]{JKL}, we just need a Borel injection $\varphi$ from $\mathcal{S}(A)$ to the class of transitive one-sided subshifts with the alphabet set $\{0,1\}$ such that for every $X_1,X_2\in \mathcal{S}(A)$, $X_1EX_2$ implies that $\varphi(X_1)$ and $\varphi(X_2)$ are conjugate.

     For $1\le i\le 6$, let $$\rho(a_i)=110100(00)^{7-i}11(00)^i$$ and $$\rho(\tilde{a}_i)=11010011(00)^{6-i}11(00)^i,$$ we also define $$\rho(\#)=(11)^501(11)^5.$$Then for $X\in \mathcal{S}(A)$, take $$\varphi(X)=\bigcup_{0\le n<22}S^n\{\rho(x(0))\rho(x(1))\rho(x(2))\cdots :x\in X\}.$$ For $x\in A^\mathbb{Z}$, we denote $\rho(x(0))\rho(x(1))\rho(x(2))\cdots$ by $\rho^+(x)$.

     Note that for every symbol $a\in A$, we have that $|\rho(a)|=22$, so we have that $S(\varphi(X))\subset \varphi(X)$. By its definition, the map $\rho^+$ is continuous from $A^\mathbb{Z}$ to $\{0,1\}^\mathbb{N}$. Then for every $X\in \mathcal{S}(A)$, $\{\rho^+(x) :x\in X\}$ is closed, $\varphi(X)$ is the union of 22 closed subsets of $\{0,1\}^\mathbb{N}$, so $\varphi(X)$ is closed. Then we get that $\varphi(X)$ is a one-sided subshift. By analyzing the occurrences of subwords $110100$ and $11011$, it's routine to check that $\{\rho(a):a\in A\}$ is uniquely readable, in other words, 
     \begin{equation}\label{1}
       \forall\, a,a',a''\in A\; \forall\,1\le n< 22\quad \rho(a)\ne(\rho(a')\rho(a''))[n,n+22) .
     \end{equation}
By (\ref{1}), for every $X\in \mathcal{S}(A)$, a word $v=b_0\cdots b_k\in A^*$ is a forbidden word of $X$ if and only if $\rho(b_0)\cdots\rho(b_n)$ is a forbidden word of $\varphi(X)$, so we have that $\varphi$ is an injection. For every $X\in \mathcal{S}(A)$, there is $x\in X$ such that $\{S^nx:n\in\mathbb{N}\}$ is dense in $X$, then $\{S^n(\rho^+(x)):n\in\mathbb{N}\}$ is dense in $\varphi(X)$, so $\varphi(X)$ is a transitive one-sided subshift. It's routine to check that the map $\varphi$ is Borel.

Then we just need to show for every $X_1,X_2\in \mathcal{S}(A)$, $X_1EX_2$ implies that $\varphi(X_1)$ and $\varphi(X_2)$ are conjugate. Recall that $f_i$ is defined in (\ref{f}). By Lemma \ref{order 2}, without loss of generality, assume that $f_1(X_1)=X_2$ and $f_1(X_2)=X_1$. \\

\noindent \textbf{Claim 2.} For $i=1,2$, $(\{S^n(\rho^+(x)):x\in X_i\})_{0\le n<22}$ is a clopen partition of $\varphi(X_i)$.\\

\begin{proof}
For every $0\le n\ne n'<22$, $i=1,2$, take $x\in X_i$, by the definition of $\rho^+$, there is $a\in A$ such that $S^{n}(\rho^+(x))[22-n,44-n)=\rho(a)$. If $S^{n}(\rho^+(x))\in\{S^{n'}(\rho^+(y)):y\in X_i\}$, by $n'\ne n$, there is $b,c\in A$ and $1\le m<22$ such that $$\rho(a)=S^{n}(\rho^+(x))[22-n,44-n)=\rho(b)\rho(c)[m,m+22),$$ contradicting (\ref{1}). So for every $0\le n\ne n'<22$, $i=1,2$ and $x\in X_i$, we have that $S^{n}(\rho^+(x))\notin\{S^{n'}(\rho^+(x)):x\in X_i\}$. Then we get that $\{S^n(\rho^+(x)):x\in X_i\}$ is disjoint from $\{S^{n'}(\rho^+(x)):x\in X_i\}$ when $0\le n\ne n'<22$, so $(\{S^n(\rho^+(x)):x\in X_i\})_{0\le n<22}$ is a partition of $\varphi(X_i)$. We have shown that $\{S^n(\rho^+(x)):x\in X_i\}$ is closed for every $1\le n<22$, so $(\{S^n(\rho^+(x)):x\in X_i\})_{0\le n<22}$ is a clopen partition of $\varphi(X_i)$.
\end{proof}

%Note that by (\ref{1}) for every $0\le n<22$, $i=1,2$, the set $\{S^n(\rho^+(x)):x\in X_i\}$ is the subset of $\varphi(X_i)$ that consists of sequences $y\in \varphi(X_i)$ satisfying $y[n,n+6)=110100$ or $y[n+8,n+13)=11011$, in particular $\{S^n(\rho^+(x)):x\in X_i\}$ is a clopen subset of $\varphi(X_i)$. Also by (\ref{1}), for every $0\le n< n'<22$, $i=1,2$, we have that $\{S^n(\rho^+(x)):x\in X_i\}$ is disjoint from $\{S^{n'}(\rho^+(x)):x\in X_i\}$. So $(\{S^n(\rho^+(x)):x\in X_i\})_{0\le n<22}$ is a clopen partition of $\varphi(X_i)$. 

Recall the definition of $g_1$ in (\ref{g}). For $k\in\mathbb{N}_+$, $0\le n<22$ and $y\in \{S^n(\rho^+(x)):x\in X_i\}$, let \begin{enumerate}
    \item [(1)] $h(y)[22k-n,22(k+1)-n)=\rho(a)$ if $y[22k-n,22(k+1)-n)=\rho(b)$, $y[22(k+1)-n,22(k+2)-n)=\rho(c)$ and $a=g_1(bc)$, where $a,b,c\in A$;
     \item [(2)] $h(y)[0,22-n)=\rho(a)[n,22)$ if $y[0,22-n)$ is a suffix of $\rho(b)$, $y[22-n,44-n)=\rho(c)$ and $a=g_1(bc)$, where $a,b,c\in A$.
\end{enumerate}
Using the same items (1) and (2) we define a map $h'$ from $\varphi(X_2)$.

We check that $h$ is a conjugacy map from $\varphi(X_1)$ to $\varphi(X_2)$. By the definition of $\rho$, for $8\le n<22$, and $1\le i\le6$, we always have $\rho(a_i)[n,22)=\rho(\tilde{a}_i)[22-n,22)$, so $h(y)[0,22-n)=y[0,22-n)$ for every $y\in\varphi(X_1)$. On the other hand, for $0\le  n<8$ and $y\in\{S^n(\rho^+(x)):x\in X_1\}$, there is a unique symbol $b\in A$ such that $y[0,22-n)=\rho(b)[n,22)$. So $h$ (and $h'$) is well defined. By definition we can see that $h$ and $h'$ are continuous and $hS=Sh$, $h'S=Sh'$. By the definition of $h$, for every $x\in X_1$, $h(\rho^+(x))=\rho^+(f_1(x))$. We know that $f_1(X_1)=X_2$, so $h(\rho^+(x))\in \varphi(X_2)$ for every $x\in X_1$, then $h$ is a map from $\varphi(X_1)$ to $\varphi(X_2)$ because $hS=Sh$. Similarly $h'$ is a map from $\varphi(X_2)$ to $\varphi(X_1)$. By the definition of $h$ and $h'$, $h'h={\rm id}_{\varphi(X_1)}$ and $hh'={\rm id}_{\varphi(X_2)}$, so we have that $h$ is a bijection from $\varphi(X_1)$ to $\varphi(X_2)$. Then we know that $h$ is a conjugacy map from $\varphi(X_1)$ to $\varphi(X_2)$. This ends the proof.
    
 \end{proof}

%\noindent\textbf{Remark:} Even if we just know a Polish group is an algebraic subgroup of some totally disconnected Polish group, by the automatic continuity of ${\rm Iso}(\mathbb{U})$ \cite{Sabok}, ${\rm Iso}(\mathbb{U})$ cannot embed to this group.

\bibliographystyle{plain}
 % We choose the "plain" reference style
\bibliography{bibliography}
\end{document}